\documentclass[11pt, reqno]{amsart}

\usepackage[margin=1.65in]{geometry}

\usepackage{ucs}
\usepackage[utf8x]{inputenc}
\usepackage{fontenc}

\usepackage{amsmath}

\usepackage{amsthm}

\usepackage{amssymb}

\usepackage{graphics}

\title{A note on fast times of Brownian motion with variable drift}
\author{}

\theoremstyle{plain}

\newtheorem{theorem}{Theorem}[section]

\newtheorem{lemma}[theorem]{Lemma}

\newtheorem{proposition}[theorem]{Proposition}

\newtheorem{corollary}[theorem]{Corollary}

\theoremstyle{definition}

\theoremstyle{remark}

\DeclareMathOperator{\zero}{\mathcal{Z}}
\newcommand{\holder}{H\"{o}lder }
\newcommand{\RR}{\mathbb{R}} 
\newcommand{\FF}{\mathfrak{F}} 
\newcommand{\GG}{\mathfrak{G}} 
\newcommand{\one}{\mathbf{1}}
\newcommand{\PP}{\mathbb{P}} 
\newcommand{\Ew}{\mathbb{E}} 

\newcommand{\JJ}{\mathcal{J}} 
\newcommand{\SSS}{\mathcal{S}} 
\newcommand{\KK}{\mathcal{K}} 

\DeclareMathOperator{\Var}{\mathbb{V}}
\begin{document}
\author[J. Ruscher]{Julia Ruscher}

\address{
Fachbereich Mathematik\\
Technische Universit\"{a}t Berlin\\
Strasse des 17. Juni 136\\
D-10623 Berlin\\
and\newline
Institute for Advanced Study\\
Einstein Drive\\
Princeton, NJ 08540
}

\email{ruscher@math.tu-berlin.de}


\begin{abstract}
A famous result of Orey and Taylor gives the Hausdorff dimension of the set of fast times, that is the set of points where linear Brownian
motion moves faster than according to the law of iterated logarithm.
In this paper we examine what happens to the set of fast times if a variable drift is added to linear Brownian motion.
In particular, we will show that the Hausdorff dimension of the set of fast times cannot be decreased by adding a function to Brownian motion.
\end{abstract}

\maketitle
\section{Introduction}

Let $B(t)$ be standard one-dimensional Brownian motion with $B(0)=0$.
In 1974 Orey and Taylor \cite{OT74} studied the so-called fast points of Brownian motion. That is, for a given $a>0$ a time $t\in[0,1]$ with \begin{align*}
\limsup_{h \downarrow 0} \frac{|B(t+h)-B(t)|}{\sqrt{2h\log{(1/h)}}} \geq a
\end{align*}
is called an \emph{$a$-fast time} of linear Brownian motion. By the law of iterated logarithm follows that the set of $a$-fast times has Lebesgue measure zero. Therefore, to quantify how often these $a$-fast times occur we use Hausdorff dimension. Orey and Taylor \cite{OT74} showed 
for every $a\in[0,1]$,
\begin{align} \label{thm: OT74}
\dim \Big\{ t\in[0,1] \Big| \limsup_{h \downarrow 0} \frac{|B(t+h)-B(t)|}{\sqrt{2h\log{(1/h)}}} \geq a  \Big\} = 1-a^2,
\end{align}
almost surely.

Khoshnevisan and Shi (\cite{KS}) extended Orey's and Taylor's results \cite{OT74} in several different ways. One of which is the intersection of the set of fast points with the zero set of Brownian motion.

\begin{theorem}[\cite{KS}] \label{thm: KS2}
Let $\zero(B):=\{ t\in(0,1] | B(t)= 0 \}$. For every $a\in(0,1]$
\[
\dim \Big\{ t\in \zero(B) \Big| \limsup_{h \downarrow 0} \frac{|B(t+h)-B(t)|}{\sqrt{2h\log{(1/h)}}} \geq a  \Big\} = \max \Big\{\frac{1}{2}-a^2,0\Big\}
\]
almost surely.
\end{theorem}

In the present note we will first give some general remarks on fast times of Brownian motion with variable drift, see section \ref{firstremarkssection}.
We can extend the result of Orey and Taylor by adding a continuous function $f$ to Brownian motion and giving a general lower bound on the Hausdorff dimension of the set of $a$-fast times. 
In particular, this theorem implies that by adding a function to Brownian motion the Hausdorff dimension of the set of $a$-fast times cannot be decreased.

\begin{theorem} \label{thm: fast times lower bound general drift}
Suppose $f \colon \mathbb{R}^+ \to \mathbb{R}$ is an arbitrary function and $X(t):= B(t)-f(t)$. For every $a\in(0,1]$
\[
\dim \Big\{ t\in [0,1] \Big| \limsup_{h \downarrow 0} \frac{|X(t+h)-X(t)|}{\sqrt{2h\log{(1/h)}}} \geq a  \Big\} \geq 1-a^2,
\]
almost surely.
\end{theorem}
An example of a function $f$ where the dimension of $a$-fast times is strictly greater than $1-a^2$ is given in the next section (see Proposition \ref{prop: fasttimeswithCantorfunction} and the subsequent example).
The following result is an upper bound analogue of Theorem \ref{thm: KS2} for $1/2$-\holder continuous functions added to one-dimensional Brownian motion. Note that the Hausdorff dimension of the zero set of the latter is $1/2$, see Corollary 1.7 of \cite{ABPR}.

\begin{theorem}\label{thm: intersection fast times and zeroset}
Suppose $f \colon \mathbb{R}^+ \to \mathbb{R}$ is a $1/2$-\holder continuous function, $X(t):= B(t)-f(t)$ and let $\zero(X):=\{ t\in(0,1] | X(t)= 0 \}$. For every $a\in(0,1]$
\[
\dim \Big\{ t\in\zero(X) \Big| \limsup_{h \downarrow 0} \frac{|X(t+h)-X(t)|}{\sqrt{2h\log{(1/h)}}} \geq a  \Big\} \leq \max \Big\{\frac{1}{2}-a^2,0\Big\}
\]
almost surely.
\end{theorem}

We will prove the upper 
using the method of \cite{KS} in section \ref{Upper bound - intersection result}
. A general lower bound for continuous functions can be given as well.

\begin{theorem}\label{thm: intersection fast times and zeroset - continuous functions}
Suppose $f \colon \mathbb{R}^+ \to \mathbb{R}$ is a continuous function, $X(t):= B(t)-f(t)$ and let $\zero(X):=\{ t\in(0,1] | X(t)= 0 \}$. For every $a\in(0,1]$
\[
\dim \Big\{ t\in\zero(X) \Big| \limsup_{h \downarrow 0} \frac{|X(t+h)-X(t)|}{\sqrt{2h\log{(1/h)}}} \geq a  \Big\} \geq \max \Big\{\frac{1}{2}-a^2, 0 
\Big\}
\]
with positive probability
.
\end{theorem}
\section{Some first remarks on fast times of Brownian motion with variable drift}\label{firstremarkssection}

By the Cameron-Martin theorem (see Theorem 1.38 in \cite{MP} or Theorem 2.2 in Chapter 8 in \cite{RY}) we see that the Theorem of Orey and Taylor, see (\ref{thm: OT74}), holds as well if we replace Brownian motion by a function $f$ added to Brownian motion where $f$ is in the Cameron-Martin space $\mathbf{D}(I)$ (integrals of functions in $\mathbf{L}^2(I)$).
We will show that the same holds for any function $f$ which is locally $1/2$-H\"{o}lder continuous. Because all functions in $\mathbf{D}(I)$ are $1/2$-H\"{o}lder continuous, this is a stronger statement than the one implied by the Cameron-Martin theorem.

\begin{corollary}\label{cor: fasttimeswithhoelder}
Let $f \colon \mathbb{R}^+ \to \mathbb{R}$ be a locally $1/2$-\holder continuous function and let $X(t):= B(t)-f(t)$. Then, for every $a\in[0,1]$
\[
\dim \Big\{ t\in[0,1] \Big| \limsup_{h \downarrow 0} \frac{|X(t+h)-X(t)|}{\sqrt{2h\log{(1/h)}}} \geq a \Big\} = 1-a^2
\]
almost surely.
\end{corollary}
\begin{proof}
By the definition of $1/2$-\holder continuity and the triangle inequality we get that for every $t\geq 0$,
\begin{align*}
\limsup_{h \downarrow 0} \frac{|B(t+h)-B(t)|}{\sqrt{2h\log{(1/h)}}} &= \limsup_{h \downarrow 0} \frac{|B(t+h)-B(t)| - |f(t+h)-f(t)|}{\sqrt{2h\log{(1/h)}}} \\ &\leq \limsup_{h \downarrow 0} \frac{|X(t+h)-X(t)|}{\sqrt{2h\log{(1/h)}}} \\ &\leq \limsup_{h \downarrow 0} \frac{|B(t+h)-B(t)| + |f(t+h)-f(t)|}{\sqrt{2h\log{(1/h)}}}
\\&= \limsup_{h \downarrow 0} \frac{|B(t+h)-B(t)|}{\sqrt{2h\log{(1/h)}}}.
\end{align*}
\end{proof}

Note that the statement of corollary \ref{cor: fasttimeswithhoelder} also holds if $f$ is not locally $1/2$-\holder continuous on a countable subset of $\RR^+$.

A natural question to ask is if we can perturb linear Brownian motion by a function such that the Hausdorff dimension of the set of $a$-fast times differs from the result (\ref{thm: OT74}). The following proposition gives a positive answer.
Also, this is an example for a function such that a strict inequality holds in Theorem \ref{thm: fast times lower bound general drift} (for some $a$).
More examples are given below.
\begin{proposition}\label{prop: fasttimeswithCantorfunction}
Let $f_\gamma \colon [0,1] \to [0,1]$ be a middle $(1-2\gamma)$-Cantor function with $\gamma <1/4$ and let $X_\gamma(t):= B(t)-f_\gamma(t)$. Then, for every $a\in[0,1]$
\begin{align*}
\dim \Big\{ t\in[0,1] \Big| \limsup_{h \downarrow 0} \frac{|X_\gamma(t+h)-X_\gamma(t)|}{\sqrt{2h\log{(1/h)}}} \geq a \Big\}  = \max\Big\{1-a^2, -\frac{\log2}{\log\gamma}\Big\},
\end{align*}
almost surely.
\end{proposition}

Note here that $-\frac{\log2}{\log\gamma}$ is also both the Hausdorff dimension of the Cantor set and the \holder exponent of the Cantor function.

\begin{proof}
For $n>0$ we call the $n$-th approximation of the Cantor set $C_{\gamma,n}$, $\mathfrak{C}_{\gamma,n}$ the set of all connected components of $C_{\gamma,n}$, and $f_{\gamma,n}$ the corresponding $n$-th approximation of the Cantor function, see e.g. \cite{ABPR}, section 3 for precise definition.
Take an arbitrary $\gamma < \gamma_1 < 1/4$. 
There is an $n_0 >0$ such that $\sum_{n \geq n_0} (2\sqrt{\gamma_1})^n \leq 1/2$.
For $n \geq n_0$ consider the interval $J_{k,n}=[k2^{-n}-\gamma_1^{n/2}/2, k2^{-n}+\gamma_1^{n/2}/2]$ and define the set $M_{n_0}= \bigcup_{n \geq n_0} \bigcup_{0 \leq k \leq 2^{n}}J_{k,n}$.
Take $t \in C_\gamma \backslash f_\gamma^{-1}(M_{n_0})$ and any $s \ne t$ in the same connected component of the interior of $C_{\gamma,n_0}$.
The largest integer $\ell$ such that both $s$ and $t$ are contained in the same interval of $\mathfrak{C}_{\gamma,\ell}$ satisfies $\ell \geq n_0$. Moreover, $|f_\gamma(s)-f_\gamma(t)| \geq \gamma_1^{(\ell+1)/2}$ and $|s-t| \leq \gamma^{\ell}$.
We see that $t$ satisfies
\[
  \limsup_{h \downarrow 0} \frac{|f_\gamma(t+h)-f_\gamma(t)|}{h^\beta} > 0.
 \]
with $\beta = \frac{\log \gamma_1 }{2 \log \gamma} < 1/2$. Hence, $t$ is an $a$-fast time of the process $X_\gamma$.

Because $\sum_{n \geq n_0} (2\sqrt{\gamma_1})^n \leq 1/2$ note that for every $n$ holds $$|C_{\gamma,n} \backslash f_\gamma^{-1}(M_{n_0}) | \geq 1/2 |C_{\gamma,n}|.$$ Therefore
the Hausdorff dimension of the fast times of the process $X_\gamma$ on the set $C_\gamma \backslash f_\gamma^{-1}(M_{n_0})$ equals the Hausdorff dimension of the Cantor set (that is $-\frac{\log2}{\log\gamma}$).

The Hausdorff dimension of fast times on the set $[0,1]\backslash C_\gamma$, that is the union of open intervals where the function $f_\gamma$ is constant, is $1-a^2$.
Note, that the set $f_\gamma^{-1}(M_{n_0}) \cap C_\gamma$ has at most the Hausdorff dimension $-\frac{\log2}{\log\gamma}$. Then, by the definition of Hausdorff dimension we see that for two sets $A$ and $B$ it holds $\dim(A\cup B) = \sup\{ \dim A, \dim B \}$. The claim follows.
\end{proof}

Note that there are functions such that for all $a>0$ the Hausdorff dimension of the set of $a$-fast times of these functions added to Brownian motion is $1$ almost surely.
For instance, Loud in \cite{Loud} constructed functions which satisfy a certain local reverse H\"{o}lder property at each point (see also the construction in \cite{MarxPiranian}).
These functions are defined as $g(t) = \sum_{k=1}^\infty g_k(t)$ where $g_k(t) = 2^{-2A\alpha k}g_0(2^{2Ak}t)$, for $0 < \alpha < 1$, a positive integer $A$ such that $2A(1-\alpha) > 1$,
 and a continuous function $g_0$ which has value $0$ at even integers, value $1$ at odd integers and is linear at all other points.
It holds that there is a positive constant $c$ such that $|g(t+h) -g(t) | > c h^{\alpha}$ for infinitely many arbitrarily small $h>0$ (see Theorem of \cite{Loud}).
Therefore, if we choose $\alpha < 1/2$, then for every $a\geq 0$,
\begin{align*}
\dim \Big\{ t\in[0,1] \Big| \limsup_{h \downarrow 0} \frac{|(B-g)(t+h)-(B-g)(t)|}{\sqrt{2h\log{(1/h)}}} \geq a \Big\} =1,
\end{align*}
almost surely.

For fractional Brownian motion Khoshnevisan and Shi (\cite{KS}) proved the following analogue result to (\ref{thm: OT74}).

\begin{theorem}[\cite{KS}] \label{thm: KS}
Let $B^{(H)} \colon \mathbb{R}^+ \to \mathbb{R}$ be a fractional Brownian motion with Hurst index $H\in]0,1[$ and $B^{(H)}(0)=0$. For every $a\in(0,1]$
\[
\dim \Big\{ t\in[0,1] \Big| \limsup_{h \downarrow 0} \frac{|B^{(H)}(t+h)-B^{(H)}(t)|}{\sqrt{2}\cdot h^{H}\sqrt{\log{(1/h)}}} \geq a  \Big\} = 1-a^2
\]
almost surely.
\end{theorem}

As before we can extend this result.

\begin{corollary}
Let $B^{(H)} \colon \mathbb{R}^+ \to \mathbb{R}$ be a fractional Brownian motion with Hurst index $H\in]0,1[$ and $B^{(H)}(0)=0$.

(i)Let $f^{(H)} \colon \mathbb{R}^+ \to \mathbb{R}$ be a locally $H$-\holder continuous function and let $X^{(H)}(t):= B^{(H)}(t)-f^{(H)}(t)$. Then, for every $a\in[0,1]$
\[
\dim \Big\{ t\in[0,1] \Big| \limsup_{h \downarrow 0} \frac{|X^{(H)}(t+h)-X^{(H)}(t)|}{\sqrt{2}\cdot h^{H}\sqrt{\log{(1/h)}}} \geq a \Big\} = 1-a^2
\]
almost surely.

(ii)Let $f_\alpha \colon [0,1] \to [0,1]$ be a middle $\alpha$-Cantor function with $\alpha \in(0,1)$ and let $X^{(H)}_\alpha(t):= B^{(H)}(t)-f_\alpha(t)$. Then, for every $a\in[0,1]$, and every $\alpha > 1 - 2^{1- \frac{1}{H}}$
\begin{multline*}
\dim \Big\{ t\in[0,1] \Big| \limsup_{h \downarrow 0} \frac{|X^{(H)}_\alpha(t+h)-X^{(H)}_\alpha(t)|}{\sqrt{2}\cdot h^{H}\sqrt{\log{(1/h)}}} \geq a \Big\} \\ = \max\Big\{1-a^2, \frac{\log2}{\log2-\log(1-\alpha)}\Big\}
\end{multline*}
almost surely.
\end{corollary}
\begin{proof}
Analogously to the proofs of Corollary \ref{cor: fasttimeswithhoelder} and Proposition \ref{prop: fasttimeswithCantorfunction}.
\end{proof}

As we have already mentioned, Khoshnevisan and Shi (\cite{KS}) also looked at the intersection set of $a$-fast times and the zero set of Brownian motion (see \ref{thm: KS2}). Unfortunately, it is not known whether an analogue statement holds for fractional Brownian motion.
(The proof cannot be adapted for fractional Brownian motion with $H \neq 1/2$ since the increments of the process are not independent.)

\section{Theorem \ref{thm: intersection fast times and zeroset}: Upper bound} \label{Upper bound - intersection result}
First we denote the set of $a$-fast times for every $a\in(0,1]$ by $F(a)$, that is
\[
F(a) := \Big\{ t\in[0,1] | \limsup_{h \downarrow 0} \frac{|X(t+h)-X(t)|}{\sqrt{2h\log{(1/h)}}} \geq a  \Big\}.
\]
By the proof of corollary \ref{cor: fasttimeswithhoelder} we see that,
\[
F(a) = \Big\{ t\in[0,1] | \limsup_{h \downarrow 0} \frac{|B(t+h)-B(t)|}{\sqrt{2h\log{(1/h)}}} \geq a  \Big\}.
\]
Further we define for every $a\in(0,1]$ and $h>0$,
\[
\FF (a,h) := \Big\{ t\in[0,1] | \sup_{t\leq s\leq t+h} |B(s)-B(t)| \geq a{\sqrt{2h\log{(1/h)}}}  \Big\}.
\]
Then, for all $0<b < a$, we have that $F(a) \subset \bigcap_{h>0}\bigcup_{0<\delta<h} \FF(b,\delta)$. Now let $I_{k,j}^\eta := [k\beta^{-\eta j}, (k+1)\beta^{-\eta j}]$ for any $\beta, \eta > 1$, all $j \geq 1$, and every integer $0 \leq k < \beta ^{\eta j}$.
For all $t\in \FF(b,\delta)$ it holds for $\delta <h <1$ with $\beta^{-j} \leq \delta \leq \beta^{1-j}$ that
\begin{multline}
 \sup_{t\leq s\leq t+\beta^{1-j}} |B(s)-B(t)| \geq b \beta^{-j/2} {\sqrt{2\log{(\beta^{j-1})}}} \\=  b \beta^{-1/2} \beta^{-(j-1)/2} {\sqrt{2\log{(\beta^{j-1})}}}.
\end{multline}
It follows $t\in \FF(b\beta^{-1/2},\beta^{1-j})$. We fix $\beta, \eta > 1$, $\theta \in ]0,1[$, then we get for any integer $i\geq 1$,
\[
F(a) \subset \bigcup_{j\geq i}\bigcup_{k\geq 1} I_{k,j}^\eta \cap \FF(\theta a\beta^{-1/2},\beta^{1-j}).
\]
%

$f$ is a $1/2$-\holder continuous function, that is $|f(t)-f(s)| \leq c_0 |t-s|^{1/2}$ for some $c_0>0$ and all $s,t \in [0,1]$. Now we will bound the probability of the event $ |B(k \beta^{-\eta j})  - f(k\beta^{-\eta j})| \leq c_1 \sqrt{\eta j \beta^{-\eta j} \log(\beta)}$ from above with $c_1:=\max\{2c_0,2\sqrt{2}\}$.
By the scaling property of Brownian motion we get,
\begin{align*}
\PP \big\{  |B(k &\beta^{-\eta j})  - f(k\beta^{-\eta j})| \leq c_1 \sqrt{\eta j \beta^{-\eta j} \log(\beta)}\big\} \\
&= \PP \big\{  B(k\beta^{-\eta j})\in [ -c_1 \sqrt{\eta j \beta^{-\eta j} \log(\beta)} + f(k\beta^{-\eta j}), \\
& \ \ \ \ \ \ \  \ \ \ \ \ \ \ \ \ \ \ \ \ \  \ \ \ \ \ \ \ \ \ \ \ \ \ \ \ \ \ \ \ \ \ \ \  c_1 \sqrt{\eta j \beta^{-\eta j} \log(\beta)} + f(k\beta^{-\eta j}) ]\big\} \\
&= \PP \big\{  B(1) \in [ -c_1 k^{-1/2} \sqrt{\eta j  \log(\beta)} + f(k\beta^{-\eta j})k^{-1/2}\beta^{\eta j/2}, \\
& \ \ \ \ \ \ \ \ \ \ \ \ \ \ \ \ \ \ \ \ \ \ \ \ \ \ \ \ \ \ \ c_1 k^{-1/2}\sqrt{\eta j \log(\beta)} + f(k\beta^{-\eta j})k^{-1/2}\beta^{\eta j/2} ]\big\}.
\end{align*}

We obtain $0 \leq |f(k\beta^{-\eta j})|k^{-1/2}\beta^{\eta j/2} \leq c_0$.
By symmetry and the unimodality property of the normal distribution, we get that
\begin{align}\label{eq: conditional}
\PP \big\{ |B(k \beta^{-\eta j})   - f(k & \beta^{-\eta j})| \leq c_1 \sqrt{\eta j \beta^{-\eta j} \log(\beta)}\big\} \nonumber\\
&\leq \PP \big\{  B(1) \in [ -c_1 k^{-1/2} \sqrt{\eta j  \log(\beta)}, c_1 k^{-1/2}\sqrt{\eta j \log(\beta)}]\big\} \nonumber\\
&\leq 2c_1k^{-1/2}\sqrt{\eta j \log(\beta)}.
\end{align}

Also, note that by Levy's modulus of continuity there exists a finite random variable $K$, depending on $\eta$ and $\beta$,
such that for all $j \geq K$ almost surely $\one_{\{ |B(k \beta^{-\eta j})  - f(k\beta^{-\eta j})| \leq c_1 \sqrt{\eta j \beta^{-\eta j} \log(\beta)} \}} \geq \one_{\{ I_{k,j}^\eta \cap \zero(X) \neq \emptyset \}}$.
Therefore, 
\[
F(a) \cap \zero(X) \subset \bigcup_{j\geq i}\bigcup_{\substack{k: |B(k \beta^{-\eta j})  - f(k\beta^{-\eta j})| \\ \leq c_1 \sqrt{\eta j \beta^{-\eta j} \log(\beta)}}  } I_{k,j}^\eta \cap \FF(\theta a\beta^{-1/2},\beta^{1-j}).
\]
The next step is to show that this is a good covering. With (\ref{eq: conditional}) we get, for any $\gamma > 0$,
\begin{align*}
 \sum_{j\geq i}  \sum_{0 \leq k < \beta^{\eta j }} &|I_{k,j}^\eta|^{\gamma} \PP(I_{k,j}^\eta \cap \FF(\theta a\beta^{-1/2},\beta^{1-j}) \neq \emptyset, I_{k,j}^\eta \cap \zero(X) \neq \emptyset) \\
 &\leq \sum_{j\geq i}\sum_{0 \leq k < \beta^{\eta j }} |I_{k,j}^\eta|^{\gamma} \PP\big(I_{k,j}^\eta \cap \FF(\theta a\beta^{-1/2},\beta^{1-j}) \neq \emptyset, \\
& \ \ \ \ \ \ \ \ \ \ \ \ \ \ \ \ \ \ \ \ \ \ \ \ \ \ \ |B(k \beta^{-\eta j})  - f(k\beta^{-\eta j})| \leq c_1 \sqrt{\eta j \beta^{-\eta j} \log(\beta)} \big) \\
 &\leq \sum_{j\geq i}\sum_{0 \leq k < \beta^{\eta j }} |I_{k,j}^\eta|^{\gamma} \PP(I_{k,j}^\eta \cap \FF(\theta a\beta^{-1/2},\beta^{1-j}) \neq \emptyset)  \\
& \ \ \ \ \ \ \ \ \ \ \ \ \ \ \ \ \ \ \ \ \ \cdot \PP\big(|B(k \beta^{-\eta j})  - f(k\beta^{-\eta j})| \leq c_1 \sqrt{\eta j \beta^{-\eta j} \log(\beta)} \big),
\end{align*}
where we used the independence of increments of Brownian motion in the last step.

In order to bound $\PP(I_{k,j}^\eta \cap \FF(\theta a\beta^{-1/2},\beta^{1-j}) \neq \emptyset)$ from above we will need the following
Lemma.
\begin{lemma}[see Lemma 3.1. of \cite{KS}]
 For all $b>0$, $0 < \varepsilon < 1$, $\eta >1$, and all $\beta > 1$, there is a $ 2 \leq J < \infty$ depending on $\varepsilon, \eta, b$ and $\beta$ such that for all
 $j \geq J$ and all $k\geq 0$,
$$\PP(I_{k,j}^\eta \cap \FF(b,\beta^{-j}) \neq \emptyset) \leq \beta^{-b^2(1-\varepsilon)j}.$$
 \end{lemma}
Hence, we obtain that for all $\mu\in]0,1[$ there is a $\infty>J\geq 2$ depending on $\mu,\eta, a, \beta$ and $\theta$ such that for $j\geq J$ and all $k>0$, $\PP(I_{k,j}^\eta \cap \FF(\theta a\beta^{-1/2},\beta^{1-j}) \neq \emptyset)$ is bounded from above by $\beta^{-\theta^2 a^2 \beta^{-1}(1-\mu)(j-1)}$. Note that $\beta^{-\theta^2 a^2 \beta^{-1}(1-\mu)(j-1)} \leq \beta^{-\theta^2 a^2 \beta^{-2}(1-\mu)j}$ for large enough $j$.
Thus, for large enough $i$,
\begin{align*}
\sum_{j\geq i}  \sum_{0 \leq k < \beta^{\eta j }} &|I_{k,j}^\eta|^{\gamma} \PP(I_{k,j}^\eta \cap \FF(\theta a\beta^{-1/2},\beta^{1-j}) \neq \emptyset)  \\
& \ \ \ \ \ \ \ \ \ \ \ \ \ \ \ \ \ \ \ \cdot \PP\big(|B(k \beta^{-\eta j})  - f(k\beta^{-\eta j})| \leq c_1 \sqrt{\eta j \beta^{-\eta j} \log(\beta)} \big) \\
&\leq \sum_{j\geq i} \beta^{-\eta \gamma j} \beta^{-\theta^2 a^2 \beta^{-2} (1-\mu)j} \Big(1+ \sum_{k=1}^{\beta^{\eta j }-1} 2c_1k^{-1/2}\sqrt{\eta j \log(\beta)} \Big)\\
&\leq \sum_{j\geq i}\beta^{-\eta \gamma j} \beta^{-\theta^2 a^2 \beta^{-2} (1-\mu)j} \big(\beta^{\eta j/2 }\cdot 2c_1\sqrt{\eta j \log(\beta)} +1\big).
\end{align*}
That means, if $\eta \gamma - \eta/2 + \theta^2 a^2 \beta^{-2}(1-\mu) > 0$, then almost surely
\begin{multline*}
\lim_{i \rightarrow \infty}\sum_{j\geq i} \sum_{0 \leq k < \beta^{\eta j }} |I_{k,j}^\eta|^{\gamma} \PP(I_{k,j}^\eta \cap \FF(\theta a\beta^{-1/2},\beta^{1-j}) \neq \emptyset)  \\
\cdot \PP\big(|B(k \beta^{-\eta j})  - f(k\beta^{-\eta j})| \leq c_1 \sqrt{\eta j \beta^{-\eta j} \log(\beta)} \big) = 0.
\end{multline*}
By letting $\mu \downarrow 0$, $\theta \uparrow 1$, $\beta \downarrow 1$ and $\eta \downarrow 1$ the claim follows.

\begin{flushright} $\blacksquare$ \end{flushright}

\section{Proof of Theorem \ref{thm: intersection fast times and zeroset - continuous functions} } \label{Lower bound - intersection result}

 In order to prove Theorem \ref{thm: intersection fast times and zeroset - continuous functions} we will give a proof of the following theorem which is an analogue of Theorem 8.1. of \cite{KS}. The statement of Theorem \ref{thm: compact set intersection} might be of independent interest.

 \begin{theorem}\label{thm: compact set intersection}
 Let $E \subset [0,1]$ be a compact set. If
 $\dim(E)> a^2 + 1/2$, then the set
 \[
 \Big\{ t\in\zero(X)\cap E  \Big| \limsup_{h \downarrow 0} \frac{|X(t+h)-X(t)|}{\sqrt{2h\log{(1/h)}}} \geq a  \Big\}
 \]
  is non-empty with positive probability.
 \end{theorem}
Now the lower bound of Theorem \ref{thm: intersection fast times and zeroset - continuous functions} follows using the following stochastic codimension argument.
For a random set $M\subset \RR_+$ the \emph{upper stochastic codimension} $\overline{\mathrm{codim}}(M)$ is defined by the smallest value $\gamma$ such that for all Borel measurable sets $G$ with $\dim(G)>\gamma$ holds that $\PP(G\cap M \neq \emptyset)>0$.
Then $\overline{\mathrm{codim}}(M) + \dim(M) \geq 1$ with positive probability, see \cite{Kho02}, p. 436 and also \cite{Kho03}, p. 238. 

In order to prove Theorem \ref{thm: compact set intersection} we need some technical lemmas. First we give some definitions.
For $\eta >0$ and an atomless probability measure $\mu$, call

\[ A_{\eta}(\mu):= \sup_{0<h\leq 1/2}\sup_{t\in[h,1-h]}\frac{\mu[t-h,t+h]}{h^\eta}.
\]
Further, define for $h>0$
$$ S_h(\mu) := \sup_{0\leq s\leq h} \int_s ^h \frac{1}{\sqrt{t-s}} d\mu(t), and $$
$$ \tilde{S}_h(\mu) := \sup_{0\leq s\leq 1} \int_s ^{(s+h\wedge 1)} \frac{1}{\sqrt{t-s}} d\mu(t). $$

The first lemma is a version of the famous Frostman's lemma.

\begin{lemma}[Frostman, cf. \cite{Kahane}, p. 130]\label{lm: Frostman}
Let $\eta >0$, and let $E \subset [0,1]$ be Borel measurable set satisfying $\eta < \dim(E)$,
then there is an atomless probability measure $\mu$ on $E$ for which
$ A_{\eta}(\mu) < \infty$.
\end{lemma}

\begin{lemma}[\cite{KS}, Lemma 8.2]\label{lm: KS Hilfslemma}
Let $\mu$ be an atomless probability measure on a compact set $E \subset [0,1]$, and for $h >0$ and $\eta > 1/2$,
$$ S_h(\mu) \leq \frac{2\exp(\eta)}{2\eta -1} A_{\eta}(\mu) h^{\eta - 1/2},$$
$$ \tilde{S}_h(\mu) \leq \frac{2\exp(\eta)}{2\eta -1} A_{\eta}(\mu) h^{\eta - 1/2}.$$
\end{lemma}

\begin{lemma}[\cite{KS}, Theorem 2.5]\label{lm: KS codimensionlemma}
Let $(E_n)$ be a countable collection of open random sets. If $\sup_{n\geq 1}\overline{\mathrm{codim}(E_n)}<1$, then
$$\overline{\mathrm{codim}}(\bigcap_{n=1}^\infty E_n)=\sup_{n\geq 1}\overline{\mathrm{codim}}(E_n).$$ 
\end{lemma}

\begin{proof}[Proof of Theorem \ref{thm: compact set intersection}]
First, for $h>0$ we define the two sets
$$ \SSS^+(h) := \{ t\in[0,1] : f(t+h) -f(t) \geq 0 \},$$
and
$$ \SSS^-(h) := \{ t\in[0,1] : f(t+h) -f(t) \leq 0 \},$$
Now for an atomless probability measure $\mu$ on $E$ let
$$ \SSS^\circ (h) := \left\{\begin{array}{l l} \SSS^-(h),& \text{ if } \int_0^1 \one_{\SSS^-(h) }(s)d\mu (s) \geq \int_0^1 \one_{ \SSS^+(h)  }(s)d\mu (s) , \\ \SSS^+(h),& \text{ if } \int_0^1 \one_{  \SSS^-(h)  }(s)d\mu (s) < \int_0^1 \one_{  \SSS^+(h)  }(s)d\mu (s) . \end{array}\right.$$
Since $\mu$ is a probability measure on the set $E \subset [0,1]$ it follows $1 \geq \int_0^1 \one_{ \SSS^\circ (h)  }d\mu (s) \geq 1/2$.

Define
\begin{align*} &{J}_{\mu}(h,a) := \int_0^1 \one_{\{  B(s) \in (f(s)-h, f(s)+h) \}} \\ &\cdot \one_{\{ B (s+h) - B(s) > a\sqrt{2h\log(1/h)} \text{ if } \SSS^\circ (h)={\SSS^-(h)}, B (s+h) - B(s) < -a\sqrt{2h\log(1/h)} \text{ if } \SSS^\circ (h) = {\SSS^+(h)}\}} \\
&\ \ \ \ \ \ \ \  \ \ \ \ \ \ \ \ \ \ \ \ \ \ \ \ \ \ \  \ \ \ \ \ \ \ \ \ \ \ \ \ \ \ \ \ \ \  \ \ \ \ \ \ \ \ \ \ \ \ \ \ \ \ \ \ \  \ \ \ \ \ \ \ \  \ \ \ \ \ \ \ \ \ \ \ \ \ \ \  \ \ \  d\mu(s).
\end{align*}
In the following we will denote the event $B (s+h) - B(s) > a\sqrt{2h\log(1/h)}$ if $\SSS^\circ (h)={\SSS^-(h)}$ and $B (s+h) - B(s) < -a\sqrt{2h\log(1/h)}$ if $\SSS^\circ (h) = {\SSS^+(h)}$ by $\KK_a (s,h)$.

For $h>0$ and $s\in[0,1]$, there are constants $C_1 >0$ and $C_2 >0$ (only depending on $\max_{x\in[0,1]} |f(x)|$) with
\begin{align}\label{probability of near zero}
 C_1 s^{-1/2}h \leq \PP(B(s) \in (f(s)-h, f(s)+h)) \leq C_2 s^{-1/2}h.
\end{align}
Note that, by the independence of increments of Brownian motion,
\begin{multline*}
 \Ew (J_{\mu}(h,a)) = \frac{1}{\sqrt{2\pi}}\int_{a\sqrt{2\log(1/h)}}^\infty \exp(-\frac{u^2}{2}) du \int_0^1 \one_{  \SSS^\circ(h) }(t)d\mu (t)
\\ \cdot
 \int_0^1 \PP(B(s) \in (f(s)-h, f(s)+h)) d\mu(s)
\end{multline*}
Applying \ref{probability of near zero} we get
\begin{align*}
 \Ew (J_{\mu}(h,a))  \geq \frac{C_1}{2\sqrt{2\pi}} h \int_{a\sqrt{2\log(1/h)}}^\infty \exp(-\frac{u^2}{2}) du \int_{h^2}^1 s^{-1/2}\mu(ds).
\end{align*}
We fix an $h'>0$, then there is a constant $c_1>0$ (depending on $\max_{x\in[0,1]} |f(x)|$) for all $0<h\leq h'$ such that
\begin{align}\label{firstmoment_fasttime}
 \Ew (J_{\mu}(h,a))  \geq c_1 h \int_{a\sqrt{2\log(1/h)}}^\infty \exp(-\frac{u^2}{2}) du.
\end{align}

Later we will apply the second moment method to $J_{\mu}(h,a)$. Therefore, we need to bound the second moment of $J_{\mu}(h,a)$ from above.
\begin{align} \label{second moment upper bound estimate with T_1 and T_2}
\Ew  &(J^2_{\mu}(h,a)) =  2 \Ew \Big[ \int_0^1 \one_{\{  B(t) \in (f(t)-h, f(t)+h) \}} \one_{\KK_a(t,h)} \nonumber\\ &\ \ \ \ \ \ \ \ \ \ \ \ \ \ \ \ \ \ \ \ \ \ \ \  \ \ \ \ \ \ \ \ \ \ \ \cdot\int_0^t \one_{\{  B(s) \in (f(s)-h, f(s)+h) \}} \one_{ \KK_a(s,h)} d\mu(s)d\mu(t) \Big ] \nonumber\\
&= \int_0^1 \one_{ \SSS^\circ(h)  }(r)d\mu (r) \cdot \sqrt{\frac{2}{\pi}}\int_{a\sqrt{2\log(1/h)}}^\infty \exp(-\frac{u^2}{2}) du \nonumber\\ 
& \ \ \ \: \cdot \Ew \Big[ \int_0^1 \one_{\{  B(t) \in (f(t)-h, f(t)+h) \}}   \int_0^t \one_{\{  B(s) \in (f(s)-h, f(s)+h) \}} \one_{ \KK_a(s,h)} d\mu(s)d\mu(t) \Big ] \nonumber
\\
&\leq \sqrt{\frac{2}{\pi}}\int_{a\sqrt{2\log(1/h)}}^\infty \exp(-\frac{u^2}{2}) du \cdot (T_1 + T_2),
\end{align}
where
\begin{align*}
T_1 &= \Ew \Big[ \int_h^1 \one_{\{  B(t) \in (f(t)-h, f(t)+h) \}}  
\cdot \int_0^{(t-h)^+} \one_{\{  B(s) \in (f(s)-h, f(s)+h) \}}  \\
& \ \ \ \ \ \ \ \  \ \ \ \ \ \ \ \ \ \ \ \ \ \ \ \ \ \ \  \ \ \ \ \ \ \ \ \ \ \ \ \ \ \ \ \ \ \ \ \ \ \ \ \ \ \ \ \ \ \ \ \ \ \ \ \ \ \cdot\one_{\KK_a(s,h)} d\mu(s)d\mu(t) \Big ], \\
T_2 &= \Ew \Big[ \int_0^1 \one_{\{  B(t) \in (f(t)-h, f(t)+h) \}}  \int_{(t-h)^+}^t \one_{\{  B(s) \in (f(s)-h, f(s)+h) \}} d\mu(s)d\mu(t) \Big ].
\end{align*}
First, we will estimate $T_1$, note that
\begin{multline*}
T_1 = \int_h^1 \int_0^{(t-h)^+} \PP  \Big( B(t) \in (f(t)-h, f(t)+h) , B(s) \in (f(s)-h, f(s)+h),\\ \KK_a(s,h) \Big)  d\mu(s)d\mu(t).
\end{multline*}
Take a $t\in[h,1]$ and an $s \in [0,t-h]$. Then we have $s\leq s+h \leq t$ and,
\begin{align*}
\PP (B(t)\in (f(t)-h&, f(t)+h) | B(r) \ \text{with} \ r \leq s+h)\\
&= \PP(B(t)-B(s+h)+B(s+h)\in (f(t)-h, f(t)+h) | B(r) \\
&\ \ \ \ \ \ \ \ \ \ \ \ \ \ \ \ \ \ \ \ \ \ \ \ \ \ \ \ \  \ \ \ \ \ \ \ \ \ \ \ \ \ \ \ \ \ \ \ \ \ \ \text{with} \ r \leq s+h)\\
&\leq \sup_{\zeta \in \RR} \PP ( B(t-s-h) + \zeta \in (f(t)-h, f(t)+h)).
\end{align*}
Since $B(t-s-h)$ is normally distributed we know by the unimodality property of the normal distribution that,
\begin{align*}
\sup_{\zeta \in \RR} \PP ( B(t-s-h) + \zeta \in (f(t)-h, f(t)+h)) \leq \PP ( B(t-s-h) \in (-h,h)).
\end{align*}
Hence, we get for $T_1$ that,
\begin{align*}
T_1 &\leq \int_h^1 \int_0^{(t-h)^+} \PP ( B(t-s-h) \in (-h,h)) \PP  \big( B(s) \in (f(s)-h, f(s)+h),\\  
&\ \ \ \ \ \ \ \ \ \ \ \ \ \ \ \ \ \ \ \ \ \ \ \ \ \ \ \ \ \ \ \
\ \ \ \ \ \ \ \ \ \ \ \ \ \ \ \ \ \ \ \ \ \ \ \ \ \ \ \ \ \ \ \ \ \ \
\ \ \ \ \  \KK_a(s,h) \big) d\mu(s)d\mu(t)\\
&\leq \int_0^1 \one_{\SSS^\circ(h)}(r)d\mu (r) \cdot \frac{1}{\sqrt{2\pi}}\int_{a\sqrt{2\log(1/h)}}^\infty \exp(-\frac{u^2}{2}) du \\ 
& \  \ \ \ \cdot \int_h^1 \int_0^{(t-h)^+} \PP ( B(t-s-h) \in (-h,h)) 
 \PP  \left( B(s) \in (f(s)-h, f(s)+h) \right) \\
&\ \ \ \ \ \ \ \ \ \ \ \ \ \ \ \ \ \ \ \ \ \ \ \ \ \ \ \ \ \ \ \
\ \ \ \ \ \ \ \ \ \ \ \ \ \ \ \ \ \ \ \ \ \ \ \ \ \ \ \ \ \ \ \ \ \ \
\ \ \ \ \  \ \ \ \ \ \ \ \ \ \ \ \ \ d\mu(s)d\mu(t).
\end{align*}
Now, by applying (\ref{probability of near zero}),
\begin{align*}
T_1 \leq c_2 h^2\int_{a\sqrt{2\log(1/h)}}^\infty \exp(-\frac{u^2}{2}) du  \int_h^1 \int_0^{(t-h)^+} \frac{1}{\sqrt{s(t-s-h)}} d\mu(s)d\mu(t),
\end{align*}
with some positive constant $c_2$ (depending on $\max_{x\in[0,1]} |f(x)|$). Further, we get
\begin{align}\label{T_1 upper bound estimate}
T_1 &\leq c_2 h^2\int_{a\sqrt{2\log(1/h)}}^\infty \exp(-\frac{u^2}{2}) du  \int_0^{1-h} \frac{1}{\sqrt{s}}  \int_{s+h}^{1} \frac{1}{\sqrt{(t-s-h)}} d\mu(t)d\mu(s) \nonumber\\
&\leq c_2 h^2\int_{a\sqrt{2\log(1/h)}}^\infty \exp(-\frac{u^2}{2}) du  \cdot S_1^2(\mu) \nonumber\\
&\leq \frac{4c_2 \exp{(2\eta)}}{(2\eta-1)^2} A_\eta^2(\mu) h^2\int_{a\sqrt{2\log(1/h)}}^\infty \exp(-\frac{u^2}{2}) du,
\end{align}
where the last step follows from Lemma \ref{lm: KS Hilfslemma}, with $\eta >1/2$.

The next step is to estimate $T_2$. Again, we use the unimodality argument as before. For all $t\geq s$ and $h>0$,

\begin{align*}
\PP ( B(t) \in  (f(t)-h  &, f(t)+h)  , B(s) \in (f(s)-h, f(s)+h) ) \\
&\leq \PP ( B(t-s) \in (-h, h)) \PP ( B(s) \in (f(s)-h, f(s)+h))\\
&\leq \PP ( B(t-s) \in (-h, h)) \PP ( B(s) \in (-h, h)).
\end{align*}

Now we can use the same calculations as in \cite{KS}, p.413 to bound $T_2$ from above. For the sake of completeness we perform them in the following. With (\ref{probability of near zero}) we get that,
\begin{align}\label{T_2 upper bound estimate}
T_2 &\leq C^2_2 h^2  \int_0^{1}  \int_{(t-h)^+}^{t} \frac{1}{\sqrt{s(t-s)}} d\mu(s)d\mu(t) \nonumber \\ \nonumber
&= C^2_2 h^2\Big[ \int_0^{h} \int_{0}^{t} \frac{1}{\sqrt{s(t-s)}} d\mu(s)d\mu(t)
+ \int_h^{1} \int_{t-h}^{t} \frac{1}{\sqrt{s(t-s)}} d\mu(s)d\mu(t)\Big]\\ \nonumber
&\leq C^2_2 h^2\Big[ \int_0^{h}  \frac{1}{\sqrt{s}}  \int_{s}^{h} \frac{1}{\sqrt{t-s}} d\mu(t) d\mu(s)   \\ \nonumber
& \ \ \ \ \ \ \ \ \ \ \ \ \ \ \ \ \ \ \ \ \ \ \ \ \ \ \ \ \ \ \ \ + \int_0^{1} \frac{1}{\sqrt{s}} \int_{s}^{(s+h)\wedge 1} \frac{1}{\sqrt{t-s}} d\mu(t) d\mu(s) \Big]\\ \nonumber
&\leq C^2_2 h^2\Big[ S_h^2(\mu) + S_1(\mu)\tilde{S}_h(\mu) \Big] \\ \nonumber
&\leq C^2_2 h^2\Big[ \frac{4 \exp{(2\eta)}}{(2\eta-1)^2} A_\eta^2(\mu) h^{2\eta -1} + \frac{4 \exp{(2\eta)}}{(2\eta-1)^2} A_\eta^2(\mu) h^{\eta -1/2} \Big]\\
&\leq \frac{8C^2_2 \exp{(2\eta)}}{(2\eta-1)^2} A_\eta^2(\mu) h^{\eta +3/2}
\end{align}
Therefore, with (\ref{second moment upper bound estimate with T_1 and T_2}), (\ref{T_1 upper bound estimate}) and (\ref{T_2 upper bound estimate}) we can now bound $\Ew (J^2_{\mu}(h,a))$ from above. There is a constant $c_3>0$ such that
\begin{align}\label{secondmoment_fasttime}
\Ew (J^2_{\mu}(h,a)) \leq \frac{c_3 \exp{(2\eta)}}{(2\eta-1)^2} A_\eta^2(\mu) (h^{\eta +3/2}\Phi +h^2\Phi ^2),
\end{align}
where $\Phi = \frac{1}{\sqrt{2\pi}}\int_{a\sqrt{2\log(1/h)}}^\infty \exp(-\frac{u^2}{2}) du $.

The next step is applying the second moment method. First, we define the four sets,
\[
 \GG (a,h) := \Big\{ t\in[0,1] \Big| \sup_{0\leq s\leq h} \frac{|X(t+s)-X(t)|}{\sqrt{2s\log{(1/s)}}} > a  \Big\},
\]
and
\[
 \GG^+ (a,h) := \Big\{ t\in [0,1] \Big| \sup_{\substack{0\leq s\leq h \\ s: t\in \SSS^-(s)}} \frac{B(t+s)-B(t)}{\sqrt{2s\log{(1/s)}}} > a  \Big\},
\]
and
\[
 \GG^- (a,h) := \Big\{ t\in[0,1] \Big| \sup_{\substack{0\leq s\leq h \\ s: t\in \SSS^+(s)}} \frac{B(t+s)-B(t)}{\sqrt{2s\log{(1/s)}}} < -a  \Big\},
\]
and $\zero_h(X):=\{t\in[0,1]: |X(t)| < h\}$.
Note that $\GG^+ (a,h) \cup \GG^- (a,h) \subset \GG (a,h)$. 
By Lemma \ref{lm: Frostman}, if $\eta< \dim(E)$, then there is an atomless probability measure $\mu$ on $E$ with $A_\eta(\mu)<\infty$. Fix such a measure $\mu$ and an $\eta$ such that $a^2+1/2<\eta<\dim(E)$.

By the Paley-Zygmund inequality, $\mathbb{P}(J_{\mu}(h,a)>0)  \geq \frac{(\mathbb{E} [J_{\mu}(h,a)])^2} {\mathbb{E} (J^2_{\mu}(h,a))}$.
Using the fact that $\int_x^{\infty} \exp{(-\frac{u^2}{2})} du \geq \frac{x}{x^2 + 1}\exp{(-\frac{x^2}{2})}$ (see for instance \cite{MP}, Lemma 12.9), we see that $\Phi \geq \frac{h^{a^2}a\sqrt{\log(1/h)}}{{\sqrt{\pi}}(2a^2\log(1/h)+1)}$.
Now, for small enough $h$ and some positive constant $c_4$ we get
\begin{align*}
\frac{(\mathbb{E} [J_{\mu}(h,a)])^2} {\mathbb{E} (J^2_{\mu}(h,a))} &\geq \frac{c_4 (2\eta-1)^2}{\exp(2\eta) A_\eta^2(\mu)}\Big[\frac{h^{\eta -1/2}}{\Phi} + 1\Big]^{-1}\\
&\geq \frac{c_4 (2\eta-1)^2}{\exp(2\eta) A_\eta^2(\mu)}\Big[h^{\eta-a^2 -1/2}\frac{\sqrt{\pi}(2a^2\log(1/h)+1)}{a\sqrt{\log(1/h)}} + 1\Big]^{-1}.
\end{align*}
Since $h^{\eta-a^2 -1/2}\frac{\sqrt{\pi}(2a^2\log(1/h)+1)}{a\sqrt{\log(1/h)}}$ goes to $0$ as $h$ goes to $0$, it follows that there is a number $\rho$ depending on $\eta$ for small enough $h$ with $\liminf_{h\rightarrow0^+}\mathbb{P}({J}_{\mu}(h,a)>0)> \rho >0$. The event ${J}_{\mu}(h,a)>0$ implies $\GG (a,h) \cap \zero_h(X) \cap E \neq \emptyset$.
Note that if $h\leq h'$, then $\{\GG (a,h) \cap \zero_h(X)\} \subset \{\GG (a,h') \cap \zero_{h'}(X)\}$. $\bigcap_{h>0}\{\GG (a,h) \cap \zero_h(X)\} \cap E$ equals to
$$\Big\{ t\in\zero(X) \cap E  \Big| \limsup_{h \downarrow 0} \frac{|X(t+h)-X(t)|}{\sqrt{2h\log{(1/h)}}} \geq a  \Big\}.$$
Observe that for every $h>0$, $\GG (a,h) \cap \zero_h(X)$ is an open subset of $[0,1]$. We can apply Lemma \ref{lm: KS codimensionlemma} since $\overline{\mathrm{codim}}(\GG (a,h) \cap \zero_h(X)) \leq \dim (E)$ for all small enough $h>0$).
It follows 
 that
$$\PP\Big( \Big\{ t\in\zero(X)\cap E  \Big| \limsup_{h \downarrow 0} \frac{|X(t+h)-X(t)|}{\sqrt{2h\log{(1/h)}}} \geq a  \Big\}\neq \emptyset\ \Big) > 0.$$

\end{proof}

\section{Proof of Theorem \ref{thm: fast times lower bound general drift}} \label{section lower bound general continuous functions without intersection}

We will need some notation first. An interval $I$ is called \textit{dyadic} if it is of the form $I=[k2^{m},(k+1)2^m]$ for some integers $k\geq 0$ and $m$.
For each positive integer $m$ let $F_m$  be the collection of dyadic intervals $[k2^{-m},(k+1)2^{-m}]$ for $k = 0, . . . , 2^m-1$ and $F$ be the union over all such collections.
For each interval $I\in F$ let $L(I)$ be a random variable that takes only the values $0$ and $1$. Define the sets 
$$\JJ_m := \bigcup_{\substack{I\in F_m: \\L(I)=1}} I,$$ and 
$$\JJ:=\bigcap_{n=1}^{\infty}\bigcup_{m=n}^{\infty} \JJ_m.$$
$\JJ$ is called \textit{limsup fractal} since $\one_\JJ=\limsup_{m\rightarrow\infty} \one_{\JJ_m}$.
In order to prove Theorem \ref{thm: fast times lower bound general drift} we will show a lower bound on the Hausdorff dimension of a certain limsup fractal. For more on this method see for instance \cite{MP}.

Fix an $\epsilon >0$ and an integer $m> 0$. For an interval $I=[t_I,s_I]$ of the form $[k2^{-m},(k+1)2^{-m}]$ we set $L(I) =1$ if $|{B}(t_I+m2^{-m}) - {B}(t_I)| > a(1+\epsilon)\sqrt{m2^{-m+1}\log(m^{-1}2^m)}$ holds.

Now we want to show that the set $\JJ$ associated with this family of random variables $\{L(I), I\in F\}$ is contained in a set of points fulfilling that at least ``half of the points'' are $a$-fast times. Then, a lower bound of the Hausdorff dimension of the set of $\JJ$ is also a lower bound of the set of $a$-fast times of the process $X$.

Note that there is a constant $c_1>0$ such that for all $s,t\in{[0,2]}$ with $|s-t|\leq h'$ with random $h'>0$,
$$ |B(s)-B(t)| \leq c_1\sqrt{|s-t|\log{\frac{1}{|s-t|}}},$$
almost surely (see Theorem 1.12 of \cite{MP}). 
Let $t\in \JJ$ and also $t\in I =[t_I,s_I] \in F_m$ with $L(I)=1$. Then, by the triangle inequality it follows
\begin{multline*}
|{B}(t + m2^{-m})-{B}(t)| \\ \geq |{B}(t_I+m2^{-m}) - {B}(t_I)| -|{B}(t + m2^{-m})-{B}(t_I+m2^{-m})| -|{B}(t_I)-{B}(t)|.
\end{multline*}
Now we see that for $m$ (larger than some random $m'>0$ and) large enough such that $a\epsilon\sqrt{2m\log(m^{-1}2^m)} \geq 2c_1\sqrt{\log{2^{m}}}$ the following inequalities hold, 
\begin{align*}
 |{B}(t_I+m2^{-m}) - {B}(t_I)| -& |{B}(t + m2^{-m})-{B}(t_I+m2^{-m})| -|{B}(t_I)-{B}(t)| \\ &\geq a(1+\epsilon)\sqrt{m2^{-m+1}\log(m^{-1}2^m)} - 2c_1\sqrt{2^{-m}\log{2^{m}}} \\ &\geq a\sqrt{m2^{-m+1}\log(m^{-1}2^m)}.
\end{align*}
This event happens for infinitely many $m$'s. 
Therefore, $t$ is an $a$-fast time of $B$.

Further, we define a process $\hat{B}$ depending on the Brownian motion $B$ by tossing a coin,
$$ \hat{B} =\left\{\begin{array}{l l} B,& \text{ with probability } 1/2, \\ -B,& \text{ with probability } 1/2 . \end{array}\right.$$ 
If the time $t$ is an $a$-fast time of $B$, then it is also an $a$-fast time of $\hat{B}$.  
Note that if $$|\hat{B}(t + m2^{-m})-\hat{B}(t)| \geq a\sqrt{m2^{-m+1}\log(m^{-1}2^m)}$$ holds, then conditional on this the event $$\hat{B}(t + m2^{-m})-\hat{B}(t) \geq a\sqrt{m2^{-m+1}\log(m^{-1}2^m)},$$ happens with probability of at least $1/2$ and
$$\hat{B}(t + m2^{-m})-\hat{B}(t) \leq -a\sqrt{m2^{-m+1}\log(m^{-1}2^m)},$$ happens with probability of at least $1/2$ as well. Therefore, we see that the event that $\hat{B}(t + m2^{-m})-\hat{B}(t) \geq a\sqrt{m2^{-m+1}\log(m^{-1}2^m)}$ if $f(t + m2^{-m})-f(t) \leq 0$ or $\hat{B}(t + m2^{-m})-\hat{B}(t) \leq -a\sqrt{m2^{-m+1}\log(m^{-1}2^m)}$ if $f(t + m2^{-m})-f(t) > 0$ happens
 with probability of at least $1/2$. Since $\hat{B}$ is also a Brownian motion, it follows that $t$ is an $a$-fast time of the process $X$ with probability of at least $1/2$.

Now set $\dim \JJ = \alpha$ (we actually know by Orey, Taylor (\ref{thm: OT74}) and Theorem \ref{thm: limsup fractal method} that $\dim \JJ = 1- a^2$ almost surely). Let $\epsilon > 0$, then there exists a probability measure $\mu$ on $\JJ$ such that the energy 
$$\Ew \Big( \int_{[0,1]} \int_{[0,1]}  \frac{1}{|x-y|^{\alpha-\epsilon}} d\mu (x) d\mu (y)  \Big) < \infty,$$
see for instance \cite{MP}, p.113 or \cite{Mattila}.
Define $$\hat{\JJ}:= \{ t \in \JJ | t \text{ is an } a\text{-fast time of } X \},$$ and a probability measure on $\hat{\JJ}$ by 
$\mu'(A) = \frac{\mu (A)}{\mu(\hat{\JJ})}$, where $A$ are measurable sets with respect to $\mu$. Note that 
$$\mu(\hat{\JJ}) = \int_{[0,1]} \PP (t \text{ is an } a\text{-fast time of } X \ | \ \FF) d \mu (t), $$
where $\FF$ is the sigma algebra of $B$. Then $\mu(\hat{\JJ}) \geq \frac{1}{2} \mu({\JJ}) = \frac{1}{2}$. Therefore,
\begin{multline*}
 \Ew \Big( \int_{[0,1]} \int_{[0,1]}  \frac{1}{|x-y|^{\alpha-\epsilon}} d\mu' (x) d\mu' (y)  \Big)\\
\leq \Ew \Big( 4 \cdot \int_{[0,1]} \int_{[0,1]}  \frac{1}{|x-y|^{\alpha-\epsilon}} d\mu (x) d\mu (y)  \Big) < \infty.
\end{multline*}
This implies $\dim \hat{\JJ} > \alpha-\epsilon$ almost surely (see Theorem 4.27 of \cite{MP}), and by letting $\epsilon \downarrow 0$, it follows that $\dim \hat{\JJ} \geq \dim \JJ$ almost surely. 

The rest of the proof is the same as in \cite{MP} and we give the details for the sake of completeness. 
The next step is to bound the first moment of $L(I)$ for $I\in F_m$ from below. Note that 
\begin{align*}
\PP(L(I) =1) &\geq \PP \big(B(t_I+m2^{-m}) - B(t_I) > a(1+\epsilon)\sqrt{m2^{-m+1}\log(m^{-1}2^m)}\big)\\
&\geq \PP (B(1) > a(1+\epsilon)\sqrt{2\log(m^{-1}2^m)})\\
&\geq 2^{-ma^2 (1+\epsilon)^3},
\end{align*}
for large enough $m$ and where the last step follows from the fact that $$\int_x^{\infty} \exp{(-\frac{u^2}{2})} du \geq \frac{x}{x^2 + 1}\exp{(-\frac{x^2}{2})},$$ (see for instance \cite{MP}, Lemma 12.9) and $\frac{a(1+\epsilon)\sqrt{2\log(m^{-1}2^m)})}{\sqrt{2\pi}(1+2a^2(1+\epsilon)^2\log(m^{-1}2^m))} \exp (-a^2 (1+\epsilon)^2 \log(m^{-1}2^m)) \geq 2^{-ma^2 (1+\epsilon)^3}$ for sufficiently large enough $m$.

In order to prove Theorem \ref{thm: fast times lower bound general drift} we will apply the following theorem.

\begin{theorem}[Theorem 10.6 of \cite{MP}]\label{thm: limsup fractal method}
Let $\JJ$ be a limsup fractal associated to the family of random variables $\{L(I), I\in F\}$. 
Suppose $p_k:=\PP(L(I) =1)$ is the same for all $I\in F_k$, and for an interval $I \in F_m$ let $M_n(I) := \sum_ {I'\subset I, I' \in F_n} L(I')$ with $m\leq n$. If there are $\eta_n \geq 1$ and $\gamma \in (0,1)$ such that
$$\Var (M_n(I)) \leq \eta_n \Ew(M_n(I)) = \eta_n p_n 2^{n-m},$$
and also
$$ \lim_{n\rightarrow\infty}  2^{(\gamma-1)n} \cdot \frac{\eta_n}{p_n} = 0,$$
then $\dim \JJ \geq \gamma$, almost surely.
\end{theorem}

In order to be able to apply Theorem \ref{thm: limsup fractal method} it is left to bound the variance $\Var (M_n(I))$ from above. To achieve this, we see that
\begin{align*}
\Ew (M^2_n(I)) &= \sum_{\substack{I_1, I_2\subset I,\\ I_1,I_2 \in F_n}} \Ew \big[ L(I_1)L(I_2) \big]\\
&\leq \sum_{\substack{I_1\subset I,\\ I_1 \in F_n}} \Big[ (2n+1)\Ew ( L(I_1)) + \Ew ( L(I_1))\sum_{\substack{I_2\subset I,\\ I_2 \in F_n}}\Ew ( L(I_2)) \Big],
\end{align*}
where we used that the random variables $L(I_1)$ and $L(I_2)$ are independent if distance of the two intervals $I_1$ and $I_2$ is at least $n2^{-n}$, and further the trivial estimate $\Ew[L(I_1)L(I_2)]\leq \Ew[L(I_1)]$. It follows that
\begin{align*}
\Var (M_n(I)) = \Ew (M^2_n(I)) - \Ew (M_n(I))^2 \leq \sum_ {\substack{I_1\subset I,\\ I_1 \in F_n}} (2n+1)p_n = 2^{n-m}(2n+1)p_n.
\end{align*}

Applying Theorem \ref{thm: limsup fractal method} for $\gamma < 1 - a^2 (1 + \epsilon)^3$, the claim follows by letting $\epsilon \downarrow 0$.

\begin{flushright} $\blacksquare$ \end{flushright}

\section*{Acknowledgments}
The author thanks gratefully Michael Scheutzow for very fruitful discussions and advice.

\end{document}